\documentclass[11pt,reqno]{amsart}
\usepackage{amssymb,amsmath,nicefrac}
\newcommand{\df}{{\;: =\;}}
\newcommand{\transp}{^{\mathsf{T}}}
\newcommand{\Exp}{\mathrm{E}}
\newcommand{\Prob}{\mathrm{P}}
\newcommand{\E}{\mathrm{e}}
\newcommand{\D}{\mathrm{d}}

\numberwithin{equation}{section}

\newtheorem{theorem}{Theorem}[section]
\newtheorem{lemma}{Lemma}[section]

\newtheorem{corollary}{Corollary}[section]

\theoremstyle{definition}
\newtheorem{definition}{Definition}[section]

\theoremstyle{remark}
\newtheorem{remark}{Remark}[section]

\begin{document}

\title[Relative Value Iteration for
Stochastic Differential Games]{Relative Value Iteration for\\
Stochastic Differential Games}

\author{Ari\ Arapostathis}
\address{Department of Electrical and Computer Engineering,
The University of Texas at Austin, 
1 University Station, Austin, TX 78712}
\email{ari@mail.utexas.edu}

\thanks{Ari Arapostathis was supported in part by ONR under the
Electric Ship Research and Development Consortium.}

\author{Vivek S.\ Borkar}
\address{Department of Electrical Engineering, Indian Institute of Technology,
Powai, Mumbai 400076, India}
\email{borkar.vs@gmail.com}
\thanks{Vivek Borkar was supported in part by Grant \#11IRCCSG014
from IRCC, IIT, Mumbai.}

\author{K.\ Suresh Kumar}
\address{Department of Mathematics, Indian Institute of Technology,
Powai, Mumbai 400076, India}
\email{suresh@math.iitb.ac.in}


\begin{abstract}
We study zero-sum stochastic differential games with player dynamics
governed by a nondegenerate controlled diffusion process.
Under the assumption of uniform stability, we establish
the existence of a solution to the Isaac's equation for the
ergodic game and characterize the optimal stationary strategies.
The data is not assumed to be bounded, nor do we assume geometric ergodicity.
Thus our results extend previous work in the literature.
We also study a relative value iteration scheme that takes the
form of a parabolic Isaac's equation.
Under the hypothesis of geometric ergodicity we show
that the relative value iteration converges to the
elliptic Isaac's equation as time goes to infinity.
We use these results to establish
convergence of the relative value iteration for
risk-sensitive control problems under an asymptotic flatness assumption.
\end{abstract}

\maketitle

\section{Introduction}
In this paper we consider a relative value iteration
for zero-sum stochastic differential games.
This relative value iteration is introduced in \cite{AriBorkar}
for stochastic control, and we follow the method introduced
in this paper.

In Section~2, we prove the existence of a solution to the Isaac's equation
corresponding to the ergodic zero-sum stochastic differential game.
We do not assume that the data or the running payoff function is bounded, nor do we
assume geometric ergodicity, so our results extend the work in
\cite{BorkarGhosh}.
In Section~3, we introduce a relative value iteration scheme for the zero-sum
stochastic differential game and prove its convergence
under a hypothesis of geometric ergodicity.
In Section~4, we apply the results from Section 3 and study a value iteration
scheme for risk-sensitive control under an asymptotic flatness assumption.

\section{Problem Description}
We consider zero-sum stochastic differential games with state dynamics
modeled by a controlled nondegenerate diffusion process $X=\{X(t): 0\le t<\infty\}$,
and subject to a long-term average payoff criterion.

\subsection{State dynamics}
Let $U_{i}$, $i=1,2\,$, be compact metric spaces and $V_{i} \;=\;{\mathcal P}(U_{i})$
denote the space of all probability measures on $U_{i}$ with Prohorov topology.
Let
$$
\Bar{b}: \mathbb{R}^d \times U_{1} \times U_{2} \to \mathbb{R}^d
\qquad\text{and}\qquad
\sigma : \mathbb{R}^d \to \mathbb{R}^{d \times d}
$$
 be measurable functions.
Assumptions on $\Bar{b}$ and $\sigma$ will be specified later.
Define $b : \mathbb{R}^d \times V_{1} \times V_{2} \to \mathbb{R}^d$ as
$$
b(x, v_{1}, v_{2}) \;\df\;\int_{U_{1}}\int_{U_{2}}\Bar{b}(x, u_{1}, u_{2})\,
v_{1}(\D{u}_{1})\, v_{2}(\D{u}_{2})\,,
$$
for $x \in \mathbb{R}^d$,
$v_{1} \in V_{1}$ and $v_{2} \in V_{2}$.
We model the controlled diffusion process $X$
via the It\^o s.d.e.
\begin{equation}\label{mainsde}
 \D{X}(t) \;=\;b\bigl(X(t), v_{1}(t), v_{2}(t)\bigr) \,\D{t}
 + \sigma\bigl(X(t)\bigr)\, \D{W}(t) \,.
\end{equation}
All processes on \eqref{mainsde} are defined in a common probability space
$(\varOmega, \, {\mathcal F}, \, P)$
which is assumed to be complete.
The process $W=\{W(t): 0\le t<\infty\}$ is an $\mathbb{R}^d$-valued standard
Wiener process which is independent of
the initial condition $X_{0}$ of \eqref{mainsde}.
Player $i$, with $i=1,2\,$, controls the dynamics $X$ through
her strategy $v_{i}(\cdot)$, a $V_{i}$-valued process which is jointly measurable in
$(t, \omega) \in [0,\infty) \times\varOmega$
and non-anticipative, i.e., for $s < t, \, W(t) - W(s)$ is independent of
$${\mathcal F}_s \;\df \,~ \text{the completion of~}
\sigma(X_{0}, v_{1}(r), v_{2}(r), W(r), r\le s)\;.
$$
 We denote the
set of all such controls (admissible controls) for player $i$
by $\mathcal{U}_{i}$, $i=1,2\,$.

\smallskip
\noindent\emph{Assumptions on the Data:}
We assume the following conditions on the coefficients $\Bar{b}$ and $\sigma$
to ensure existence of a unique solution to \eqref{mainsde}.

\begin{itemize}
\item[(A1)]
The functions $\Bar{b}$ and $\sigma$ are
locally Lipschitz continuous in $x\in\mathbb{R}^d$, uniformly
over $(u_{1},u_{2})\in U_{1}\times U_{2}\,$, and have at most a linear growth
rate in $x\in\mathbb{R}^{d}$, i.e., for some constant $\kappa$,
\begin{equation*}
\|\Bar{b}(x,u_{1},u_{2})\|^{2}+ \|\sigma(x)\|^{2}\le \kappa
\bigl(1 + \|x\|^{2}\bigr) \qquad
\forall (x,u_{1},u_{2})\in\mathbb{R}^{d}\times U_{1}\times U_{2}\;,
\end{equation*}
where $\|\sigma\|^{2}\df
\mathrm{trace}\left(\sigma\sigma^{\mathsf{T}}\right)$,
with $\transp$ denoting the transpose.
Also $\Bar{b}$ is continuous.

\smallskip
\item[(A2)]
For each $R>0$ there exists a constant $\kappa(R) > 0$ such that
$$
z\transp a(x) z \;\ge\; \kappa(R) \|z\|^2
\qquad \text{for~all~} \|x\|\le R \text{~and~} z \in \mathbb{R}^d \,,
$$
where $a\;\df\; \sigma \sigma\transp$.
\end{itemize}

\begin{definition}
For $f \in C^2(\mathbb{R}^d)$ define
\begin{equation*}
 \Bar{L} f (x, u_{1}, u_{2}) \;\df \;
 \Bar{b}(x, u_{1}, u_{2}) \cdot \nabla f(x) \, + \,
 \frac{1}{2}\,\mathrm{tr}\bigl(a(x) \nabla^2 f(x)\bigr)
\end{equation*}
for $x \in \mathbb{R}^d$ and $(u_{1},u_{2})\in U_{1}\times U_{2}\,$.
Also define the \emph{relaxed extended controlled generator} $L$ by
\begin{equation*}
 L f (x, v_{1}, v_{2}) \;\df \;\int_{U_{1}}\int_{U_{2}} L f(x,u_{1}, u_{2})\,
 v_{1}(\D{u}_{1}) \,v_{2}(\D{u}_{2})\,,
\quad f \in C^2(\mathbb{R}^d)\,,
\end{equation*}
for $x \in \mathbb{R}^d$ and $(v_{1},v_{2})\in V_{1}\times V_{2}\,$.

We denote the set of all stationary Markov strategies of player $i$
by ${\mathcal M}_{i}\,$, $i=1,2\,$.
\end{definition}

\subsection{Zero-sum ergodic game}
Let $\Bar{h} : \mathbb{R}^d \times U_{1} \times U_{2} \to [0,\infty)$ be a continuous
function, which is also locally Lipschitz continuous in its first argument.
We define the \emph{relaxed running payoff function}
$h : \mathbb{R}^d \times V_{1} \times V_{2} \to [0,\infty)$ by
$$
h(x,v_{1},v_{2}) \;\df\;\int_{U_{1}}\int_{U_{2}}
\Bar{h}(x, u_{1}, u_{2}) \, v_{1}(\D{u}_{1})\, v_{2}(\D{u}_{2}) \,.
$$
Player 1 seeks to maximize the average payoff given by
\begin{equation}\label{avgpayoff}
 \liminf_{T \to \infty}\; \frac{1}{T}\;
 E_{x} \biggl[\int^T_{0} h(X(t), v_{1}(t), v_{2}(t)) \,\D{t} \biggr]
\end{equation}
over all admissible controls $v_{1}\in\mathcal{U}_{1}\,$, while Player 2 seeks
to minimize \eqref{avgpayoff} over all $v_{2}\in\mathcal{U}_{2}$.
Here $E_{x}$ is the expectation operator corresponding to the probability
measure on the canonical space of the process starting at $X(0)=x$.

Since we shall analyze the average payoff as a limiting case of the discounted
payoff in the `vanishing discount' limit, we shall also consider the infinite
horizon discounted payoff
\begin{equation*}
E_{x} \biggl[\int_{0}^{\infty}\E^{-\alpha t}h(X(t), v_{1}(t), v_{2}(t))\,\D{t}
\biggr]\,,
\end{equation*}
where $\alpha > 0$ is the discount factor.

\smallskip
\noindent\emph{Assumptions on Ergodicity:}
We consider the following ergodicity assumptions:
\begin{itemize}
\item[(A3)]
There exist a positive inf-compact function
$\mathcal{V} \in C^2(\mathbb{R}^d)$
and positive constants $k_{0}$, $k_{1}$ and $k_{2}$ such that
\begin{align*}
 \Bar{L} \mathcal{V}(x, u_{1}, u_{2}) &\;\le\; k_{0} - 2 k_{1} \mathcal{V}(x)\;,
 \\[5pt]
\max_{u_{1} \in U_{1}, u_{2} \in U_{2}} \Bar{h}(x, u_{1}, u_{2})\, &\;\le\; k_{2}\,
\mathcal{V}(x)
\end{align*}
for all $(u_{1},u_{2})\in U_{1}\times U_{2}\,$, and
$x \in \mathbb{R}^d$.
Without loss of generality we assume $\mathcal{V}\ge1$.

\medskip
\item[(A3$\mspace{2mu}^\mathbf\prime$)]
There exist nonnegative inf-compact functions $\mathcal{V} \in C^2(\mathbb{R}^d)$ and
$g \in C(\mathbb{R}^d)$, and positive constants $k_{0}$ and $k_{2}$ such that
\begin{align*}
 \Bar{L} \mathcal{V}(x, u_{1}, u_{2}) \, &\;\le\; k_{0} - g(x)\;, \\[5pt]
 \max_{u_{1} \in U_{1}, u_{2} \in U_{2}} \Bar{h}(x, u_{1}, u_{2}) &\;\le\;
k_{2}\, g(x)
\end{align*}
for all $(u_{1},u_{2})\in U_{1}\times U_{2}\,$, and $x \in \mathbb{R}^d$.
Also,
\begin{equation*}
\frac{\max_{u_{1} \in U_{1}, u_{2} \in U_{2}}\;\Bar{h}(x, u_{1}, u_{2})}{g(x)}
\xrightarrow[\|x\|\to\infty]{} 0\,.
\end{equation*}
Without loss of generality we assume $\mathcal{V}\ge1$ and $g\ge1$.
\end{itemize}

In this section we use assumption (A3$\mspace{2mu}^\mathbf\prime$), while
in Section~3 we employ (A3) which is stronger and equivalent to geometric
ergodicity in the time-homogeneous Markov case.

For the uncontrolled (i.e., Markov) case, (A3$\mspace{2mu}^\mathbf\prime$) is the
so called `g-norm ergodicity' in the terminology of \cite{MT-III}
which implies, in addition to convergence of laws to a unique stationary
distribution, convergence of $\frac{1}{t}\int_{0}^{t}E[f(X(s))]\,\D{s}$
to the corresponding stationary
expectation as $t\uparrow\infty$ for all $f$ with growth rate at most that of
$g$ and vice versa.
Assumption (A3) corresponds to the same with $h = V$ and implies in particular
exponential convergence to stationary averages (and vice versa).
This is the so called geometric ergodicity.
When (A3$\mspace{2mu}^\mathbf\prime$) holds in the controlled case, it implies in
particular tightness of stationary distributions attainable under
stationary Markov controls.
In fact this condition is necessary and sufficient.
See \cite[Lemma~3.3.4]{AriBorkarGhosh} for this and other
equivalent characterizations.
Thus (A3$\mspace{2mu}^\mathbf\prime$) is the best possible condition
for uniform stability in this sense.
While the results of \cite{AriBorkar} can be extended to control problems
when instability is possible but is penalized by the cost structure,
this does not extend naturally to the zero sum game, because what is penalty
for one agent is a reward for the other.

We start with a theorem which characterizes the
value of the game under a discounted
infinite horizon criterion.
For this we need the following notation:
For a continuous function $\mathcal V\colon\mathbb{R}^{d}\to(0,\infty)$,
$C_{\mathcal{V}}(\mathbb{R}^{d})$ denotes the space of functions
in $C(\mathbb{R}^{d})$ satisfying
$\sup_{x\in\mathbb{R}^{d}}\;
\left|\frac{f(x)}{\mathcal{V}(x)}\right|<\infty$.
This is a Banach space under the norm
$$\|f\|_{\mathcal{V}}\;\df\;\sup_{x\in\mathbb{R}^{d}}\;
\left|\frac{f(x)}{\mathcal{V}(x)}\right|\,.$$

\begin{theorem}\label{thm2.1}
Assume \textup{(A1), (A2)} and \textup{(A3$\mspace{2mu}^\prime$)}.
For $\alpha > 0$, there exists a solution
$\varphi_{\alpha} \in C_{\mathcal{V}}(\mathbb{R}^d) \cap C^2(\mathbb{R}^d)$
to the p.d.e.
\begin{equation}\label{Isaac-discounted}
\begin{split}
\alpha \psi_{\alpha}(x) &\;=\; \min_{v_{2}\in V_{2}}\; \max_{v_{1}\in V_{1}}\;
\bigl[L \psi_{\alpha}(x, v_{1}, v_{2}) + h(x,v_{1},v_{2})\bigr]\\[5pt]
&\;=\; \max_{v_{1}\in V_{1}}\; \min_{v_{2}\in V_{2}}\;
\bigl[L \psi_{\alpha}(x, v_{1}, v_{2}) + h(x,v_{1},v_{2}) \bigr]
\end{split}
\end{equation}
and is characterized by
\begin{align*}
\psi_{\alpha}(x) &\;=\; \sup_{v_{1}\in\, \mathcal{U}_{1}}\;
\inf_{v_{2}\in\, \mathcal{U}_{2}}\;
 \Exp_{x} \biggl[\int^{\infty}_{0} \E^{- \alpha t}
h\bigl(X(t), v_{1}(t), v_{2}(t)\bigr) \,\D{t}\biggr]\\[5pt]
&\;=\; \inf_{v_{2}\in\, \mathcal{U}_{2}}\; \sup_{v_{1}\in\, \mathcal{U}_{1}}\;
\Exp_{x} \biggl[\int^{\infty}_{0} \E^{- \alpha t}
h\bigl(X(t), v_{1}(t), v_{2}(t)\bigr) \,\D{t} \biggr] \,.
\end{align*}
\end{theorem}

\begin{proof}
Let $B_{R}$ denote the open ball of radius $R$ centered at the origin
in $\mathbb{R}^d$.
The p.d.e.
\begin{equation}\label{upperIsaacdisc-dir}
\begin{split}
\alpha \varphi^R_{\alpha}(x) &\;=\;
\min_{v_{2}\in V_{2}}\; \max_{v_{1}\in V_{1}}\;
\bigl[L \varphi^R_{\alpha}(x, v_{1}, v_{2}) + h(x,v_{1},v_{2}) \bigr]\,,\\[5pt]
\varphi^R_{\alpha} &\;=\; 0\qquad \text{on~}\partial B_{R}
\end{split}
\end{equation}
has a unique solution $\varphi^R_{\alpha}$ in
$C^2(B_{R}) \cap C (\overline{B_{R}})$,
see \cite[Theorem~15.12, p.~382]{GilbargTrudinger}.
Since
\begin{equation*}
\min_{v_{2}\in V_{2}}\; \max_{v_{1}\in V_{1}}\;
\bigl[L \varphi^R_{\alpha}(x, v_{1}, v_{2}) + h(x,v_{1},v_{2}) \bigr] \;=\;
\max_{v_{1}\in V_{1}}\; \min_{v_{2}\in V_{2}}\;
\bigl[L \varphi^R_{\alpha}(x, v_{1}, v_{2}) + h(x,v_{1},v_{2}) \bigr]\,,
\end{equation*}
it follows that $\varphi^R_{\alpha} \in C^2(B_{R}) \cap C (\overline{B_{R}})$
is also a solution to
\begin{equation}\label{lowerIsaacdisc-dir}
\begin{split}
\alpha \varphi^R_{\alpha}(x) &\;=\;
\max_{v_{1}\in V_{1}}\; \min_{v_{2}\in V_{2}}\;
\bigl[L \varphi^R_{\alpha}(x, v_{1}, v_{2}) + h(x,v_{1},v_{2}) \bigr]\,,\\[5pt]
\varphi^R_{\alpha} &\;=\; 0\qquad \text{on~}\partial B_{R}\,.
\end{split}
\end{equation}

Let $v^R_{1 \alpha}\colon B_{R}\to V_{1}$ be a measurable selector for the
maximizer in \eqref{lowerIsaacdisc-dir} and $v^R_{2 \alpha}\colon B_{R}\to V_{2}$
be a measurable selector for the minimizer in \eqref{upperIsaacdisc-dir}.
If we let
\begin{equation*}
F\bigl(x,v_{1};\varphi^R_{\alpha}\bigr) \;\df\; \min_{v_{2}\in V_{2}}\;
\bigl[L \varphi^R_{\alpha}(x, v_{1}, v_{2}) + h(x,v_{1},v_{2}) \bigr]\;,
\end{equation*}
then $(x,v_{1})\mapsto F\bigl(x,v_{1};\varphi^R_{\alpha}\bigr)$ is continuous
and also Lipschitz in $x$,
and $\varphi^R_{\alpha}$ satisfies
\begin{align*}
\alpha \varphi^R_{\alpha}(x) &\;=\;
F\bigl(x,v^R_{1 \alpha}(x);\varphi^R_{\alpha}\bigr)
\\[5pt]
&\;=\;
\min_{v_{2}\in V_{2}}\; \bigl[L \varphi^R_{\alpha}(x, v^R_{1 \alpha}(x), v_{2})
+ h(x, v^R_{1 \alpha}(x), v_{2} ) \bigr]\,, \\[5pt]
\varphi^R_{\alpha} &\;=\; 0 \qquad \text{on~}\partial B_{R}\;.
 \end{align*}
By a routine application of It\^o's formula, it follows that
\begin{equation}\label{rep1}
\varphi^R_{\alpha}(x) \;=\;\inf_{v_{2}\in\, \mathcal{U}_{2}}\;
\Exp_{x} \biggl[\int^{\tau_{R}}_{0} \E^{- \alpha t}
h\bigl(X(t), v^R_{1 \alpha}(X(t)),v_{2}(t)\bigr) \,\D{t} \biggr] \,,
\end{equation}
where
$$\tau_{R} \;\df\;\inf\; \{ t \ge 0 : \|X(t) \| \ge R \}$$
and $X$ is the solution to \eqref{mainsde} corresponding to the control pair
$(v^R_{1 \alpha}, v_{2})$, with $v_{2}\in \mathcal{U}_{2}\,$.

Repeating the above argument with the outer minimizer
$v^R_{2 \alpha}$ of \eqref{upperIsaacdisc-dir}, we similarly obtain
\begin{equation}\label{rep2}
\varphi^R_{\alpha}(x) \;=\;\sup_{v_{1}\in\, \mathcal{U}_{1}}\;
\Exp_{x} \biggl[\int^{\tau_{R}}_{0} \E^{- \alpha t}
h\bigl(X(t), v_{1}(t), v^R_{2 \alpha}(X(t))\bigr) \,\D{t} \biggr] \,.
\end{equation}
Combining \eqref{rep1} and \eqref{rep2}, we obtain
\begin{multline*}
\inf_{v_{2}\in\, \mathcal{U}_{2}}\; \sup_{v_{1}\in\, \mathcal{U}_{1}}\;
\Exp_{x}
\biggl[\int^{\tau_{R}}_{0} \E^{-\alpha t} h\bigl(X(t),v_{1}(t),v_{2}(t)\bigr)\,\D{t}
\biggr]
\;\le\;\varphi^R_{\alpha}(x) \\[5pt]
\;\le\; \sup_{v_{1}\in\, \mathcal{U}_{1}}\;
\inf_{v_{2}\in\, \mathcal{U}_{2}}\; \Exp_{x} \biggl[\int^{\tau_{R}}_{0}
\E^{- \alpha t} h\bigl(X(t), v_{1}(t), v_{2}(t)\bigr) \,\D{t} \biggr]\,,
\end{multline*}
which implies that
\begin{align*}
\varphi^R_{\alpha}(x) &\;=\; \sup_{v_{1}\in\, \mathcal{U}_{1}}\;
\inf_{v_{2}\in\, \mathcal{U}_{2}}\; \Exp_{x} \biggl[\int^{\tau_{R}}_{0}
\E^{- \alpha t} h\bigl(X(t), v_{1}(t), v_{2}(t)\bigr) \,\D{t} \biggr]\\[5pt]
&\;=\; \inf_{v_{2}\in\, \mathcal{U}_{2}}\; \sup_{v_{1}\in\, \mathcal{U}_{1}}\;
\Exp_{x} \biggl[\int^{\tau_{R}}_{0} \E^{- \alpha t}
h\bigl(X(t), v_{1}(t), v_{2}(t)\bigr) \,\D{t}
\biggr] \,.
\end{align*}

It is evident that
$\varphi^R_{\alpha}(x) \;\le\; \Tilde{\psi}_{\alpha}(x)$, $x \in \mathbb{R}^d$,
where
$$
\Tilde{\psi}_{\alpha}(x) \;\df\; \sup_{v_{1}\in\, \mathcal{U}_{1}}\;
\inf_{v_{2}\in\, \mathcal{U}_{2}}\;
\Exp_{x} \biggl[\int^{\infty}_{0}
\E^{- \alpha t} h\bigl(X(t), v_{1}(t), v_{2}(t)\bigr) \,\D{t}\biggr]\,,
\quad x \in \mathbb{R}^d\,.
$$
Also $\varphi^R_{\alpha}$ is nondecreasing in $R$.
By Assumption~(A3$\mspace{2mu}^\mathbf\prime$), it follows that
$$
\Tilde{\psi}_{\alpha}(x) \;\le\;
k_{2}\, \Exp_{x} \biggl[\int^{\infty}_{0} \E^{- \alpha t}
g\bigl(X(t)\bigr) \,\D{t}\biggr] \,,
$$
where $X$ is a solution to \eqref{mainsde} corresponding to
some stationary Markov control pair.
Since the function $x \mapsto \Exp_{x} \bigl[\int^{\infty}_{0} \E^{- \alpha t}
g\bigl(X(t)\bigr) \,\D{t}\bigr]$
is continuous, it follows that
$\Tilde{\psi}_{\alpha} \in L^{p}_{loc}(\mathbb{R}^d)$ for $1 < p < \infty$.

Bene{\v{s}}' measurable selection theorem \cite{Benes} asserts that
there exist controls
$(v^R_{1 \alpha}, \, v^R_{2 \alpha})\in {\mathcal M}_{1} \times {\mathcal M}_{2}$
which realize the minimax in
\eqref{upperIsaacdisc-dir}--\eqref{lowerIsaacdisc-dir}, i.e.,
for all $x \in B_{R}$ the following holds:
\begin{equation*}
\max_{v_{1}\in V_{1}}\; \min_{v_{2}\in V_{2}}\;
\bigl[L \varphi^R_{\alpha}(x, v_{1}, v_{2}) + h(x,v_{1},v_{2}) \bigr]
\;= \;
L \varphi^R_{\alpha} \bigl(x, v^R_{1 \alpha}(x), v^R_{2 \alpha}(x) \bigr)
+ h\bigl(x, v^R_{1 \alpha}(x), v^R_{2 \alpha}(x) \bigr) \,.
\end{equation*}
Hence $\varphi^R_{\alpha} \in C^2(B_{R}) \cap C (\overline{B_{R}})$ is a solution to
$$
\alpha \varphi^R_{\alpha} (x) \;=\;
L \varphi^R_{\alpha} \bigl(x, v^R_{1 \alpha}(x), v^R_{2 \alpha}(x) \bigr) +
h \bigl(x, v^R_{1 \alpha}(x), v^R_{2 \alpha}(x)\bigr)\,,\quad x \in B_{R} \,.
$$
Hence by \cite[Lemma~A.2.5, p.~305]{AriBorkarGhosh}, for each $1 < p < \infty$
and $R' > 2R$, we have
\begin{align*}
\bigl\|\varphi^{R'}_{\alpha} \bigr\|_{W^{2,p}(B_{R})} & \,\le\,
K_{1} \Bigl( \bigl\|\varphi^{R'}_{\alpha} \bigr\|_{L^{p}(B_{2R})}
+ \bigl\|L \varphi^{R'}_{\alpha} - \alpha \varphi^{R'}_{\alpha}
\bigr\|_{L^p(B_{2R})}\Bigr)
\\[5pt]
& \,\le\, K_{1} \Bigl( \bigl\|\Tilde{\psi}_{\alpha} \bigr\|_{L^{p}(B_{2R} \bigr)} +
\bigl\|h \bigl(\,\cdot\,, v^{R'}_{1 \alpha}(\,\cdot\,),
v^{R'}_{2 \alpha}(\,\cdot\,)\bigr) \bigr\|_{L^p(B_{2R})} \Bigr)
\\[5pt]
& \,\le\, K_{1} \Bigl( \bigl\|\Tilde{\psi}_{\alpha}
\bigr\|_{L^{p}(B_{2R})} + K_{2}(R) |B_{2R}|^{\nicefrac{1}{p}} \Bigr) \,,
\end{align*}
where $K_{1} >0$ is a constant independent of $R'$ and $K_{2}(R)$ is a constant
depending only on the bound of $h$ on $B_{2R}$.
Using standard approximation arguments involving Sobolev imbedding theorems,
see \cite[p.~111]{AriBorkarGhosh},
it follows that there exists $\psi_{\alpha} \in W^{2, p}_{loc}(\mathbb{R}^d)$
such that
$\varphi^{R}_{\alpha} \uparrow \psi_{\alpha}$ as $R\uparrow\infty$
and $\psi_{\alpha}$ is a solution to
$$
\alpha \psi_{\alpha}(x) \;=\;
\max_{v_{1}\in V_{1}}\; \min_{v_{2}\in V_{2}}\;
\bigl[L \psi_{\alpha}(x, v_{1}, v_{2}) + h(x,v_{1},v_{2}) \bigr] \,.
$$
By standard regularity arguments, see \cite[p.~109]{AriBorkarGhosh},
one can show that
$\psi_{\alpha} \in C^{2, r}(\mathbb{R}^d)$, $0 < r < 1$.
Also using the minimax condition, it follows that
$\psi_{\alpha} \in C^{2, r}(\mathbb{R}^d)$, $0 < r < 1$, is a solution to
\begin{align*}
 \alpha \psi_{\alpha}(x) \;&=\;
\min_{v_{2}\in V_{2}}\; \max_{v_{1}\in V_{1}}\; \bigl[L \psi_{\alpha}(x, v_{1}, v_{2})
+ h(x,v_{1},v_{2}) \bigr] \\[5pt]
\;&=\; \max_{v_{1}\in V_{1}}\; \min_{v_{2}\in V_{2}}\;
\bigl[L \psi_{\alpha}(x, v_{1}, v_{2}) + h(x,v_{1},v_{2}) \bigr]\,.
 \end{align*}
Let $v^{\alpha}_{1}\in{\mathcal M}_{1}$ and $v^{\alpha}_{2}\in{\mathcal M}_{2}$
be an outer maximizing and an outer minimizing selector
for \eqref{Isaac-discounted}, respectively, corresponding to $\psi_{\alpha}$
given above.
Then $\psi_{\alpha}$ satisfies the p.d.e.
$$
\alpha \psi_{\alpha}(x) \;=\;\max_{v_{1} \in V_{1}}\;
\bigl[L \psi_{\alpha} \bigl(x, v_{1}, v^{\alpha}_{2}(x)\bigr) +
h\bigl(x, v_{1}, v^{\alpha}_{2}(x) \bigr) \bigr]\,.
$$
For $v_{1}\in \mathcal{U}_{1}$, let $X$ be the solution to
\eqref{mainsde} corresponding to $(v_{1}, v^{\alpha}_{2})$ and the
initial condition $x \in \mathbb{R}^d$.
Applying the It\^o--Dynkin formula, we obtain
$$
\Exp_{x} \bigl[\E^{- \alpha \tau_{R}} \psi_{\alpha}(X(\tau_{R})) \bigr]
- \psi_{\alpha}(x)
\,\le\, - \Exp_{x} \biggl[\int^{\tau_{R}}_{0} \E^{- \alpha t}
h\bigl(X(t), v_{1}(t),v^{\alpha}_{2}(X(t))\bigr) \,\D{t}\biggr]\,.
$$
Since $\psi_{\alpha} \ge 0$, we have
$$
\psi_{\alpha}(x) \;\ge\; \Exp_{x} \biggl[\int^{\tau_{R}}_{0} \E^{- \alpha t}
h\bigl(X(t), v_{1}(t), v^{\alpha}_{2}(X(t))\bigr) \,\D{t}\biggr]\,.
$$
Using Fatou's lemma we obtain
\begin{equation}\label{rep-discounted1a}
\psi_{\alpha}(x) \;\ge\; \Exp_{x} \biggl[\int^{\infty}_{0}
\E^{- \alpha t} h\bigl(X(t), v_{1}(t), v^{\alpha}_{2}(X(t))\bigr) \,\D{t}\biggr]\,.
\end{equation}
Therefore
\begin{equation}\label{rep-discounted1}
\psi_{\alpha}(x) \;\ge\;
\sup_{v_{1}\in\, \mathcal{U}_{1}}\;
\Exp_{x} \biggl[\int^{\infty}_{0} \E^{- \alpha t}
h\bigl(X(t), v_{1}(t), v^{\alpha}_{2}(X(t))\bigr) \,\D{t}\biggr]\,.
\end{equation}
Similarly, for $v_{2}\in \mathcal{U}_{2}$, let $X$ be the solution to
\eqref{mainsde} corresponding to $(v^{\alpha}_{1}, v_{2})$ and the initial condition
$x \in \mathbb{R}^d$.
By applying the It\^o--Dynkin formula, we obtain
$$
\Exp_{x} \bigl[\E^{- \alpha \tau_{R}} \psi_{\alpha}(X(\tau_{R})) \bigr]
- \psi_{\alpha}(x)
\,\ge\,
- \Exp_{x} \biggl[\int^{\tau_{R}}_{0} \E^{- \alpha t}
h\bigl(X(t),v^{\alpha}_{1}(X(t)),v_{2}(t)\bigr) \,\D{t}\biggr]\,.
$$
Hence
$$
\psi_{\alpha}(x) \;\le\; \Exp_{x} \biggl[\int^{\infty}_{0}
\E^{- \alpha t} h\bigl(X(t),v^{\alpha}_{1}(X(t)),v_{2}(t)\bigr) \,\D{t}
\biggr]
\, + \, \Exp_{x}\bigl[\E^{- \alpha \tau_{R}} \psi_{\alpha}(X(\tau_{R}))\bigr] \,.
$$
By \cite[Remark A.3.8, p.~310]{AriBorkarGhosh}, it follows that
$$
\lim_{R \uparrow \infty }\;
\Exp_{x} \bigl[\E^{- \alpha \tau_{R}} \psi_{\alpha}(X(\tau_{R})) \bigr] \;=\; 0 \,.
$$
Hence, we have
\begin{equation}\label{rep-discounted2a}
\psi_{\alpha}(x) \;\le\; \Exp_{x} \biggl[\int^{\infty}_{0} \E^{- \alpha t}
h\bigl(X(t),v^{\alpha}_{1}(X(t)),v_{2}(t)\bigr) \,\D{t}\biggr]\,.
\end{equation}
Therefore
\begin{equation}\label{rep-discounted2}
\psi_{\alpha}(x) \;\le\;
\inf_{v_{2}\in\, \mathcal{U}_{2}}\; \Exp_{x} \biggl[\int^{\infty}_{0} \E^{- \alpha t}
h\bigl(X(t),v^{\alpha}_{1}(X(t)),v_{2}(t)\bigr) \,\D{t}\biggr]\,.
\end{equation}
By \eqref{rep-discounted1} and \eqref{rep-discounted2}, we obtain
\begin{equation}\label{rep-discounted3}
\psi_{\alpha}(x) \;= \; \Exp_{x} \biggl[\int^{\infty}_{0} \E^{- \alpha t}
h\bigl(X(t),v^{\alpha}_{1}(X(t)),v^{\alpha}_{2}(X(t))\bigr) \,\D{t} \biggr]\,.
\end{equation}
Also by \eqref{rep-discounted1a} and \eqref{rep-discounted2a} we have
\begin{multline*}
\inf_{v_{2}\in\, \mathcal{U}_{2}}\; \sup_{v_{1}\in\, \mathcal{U}_{1}}\;
\Exp_{x} \biggl[\int^{\infty}_{0} \E^{- \alpha t}
h\bigl(X(t),v_{1}(t),v_{2}(t)\bigr) \,\D{t}\biggr]
\;\le\; \psi_{\alpha}(x) \\[5pt]
\le\; \sup_{v_{1}\in\, \mathcal{U}_{1}}\;
\inf_{v_{2}\in\, \mathcal{U}_{2}}\; \Exp_{x} \biggl[\int^{\infty}_{0} \E^{- \alpha t}
h\bigl(X(t),v_{1}(t),v_{2}(t)\bigr) \,\D{t}\biggr]\,.
\end{multline*}
This implies the desired characterization.
\end{proof}

\begin{remark}
Using Theorem~\ref{thm2.1}, one can
easily show that any pair of measurable outer maximizing and
outer minimizing selectors
of \eqref{Isaac-discounted} is a saddle point equilibrium for the
stochastic differential game with state
dynamics given by \eqref{mainsde} and with a discounted criterion under
the running payoff function $h$.
\end{remark}

\begin{theorem}\label{thm2.2}
Assume \textup{(A1), (A2)} and \textup{(A3$\mspace{2mu}^\prime$)}.
Then there exists a solution
$(\beta, \varphi^*)
\in \mathbb{R} \times C_{\mathcal{V}}(\mathbb{R}^d) \cap C^2(\mathbb{R}^d)$
to the Isaac's equation
\begin{equation}\label{Isaac}
\begin{split}
\beta &\;=\; \min_{v_{2}\in V_{2}}\; \max_{v_{1}\in V_{1}}\;
\bigl[L \varphi^*(x, v_{1}, v_{2}) + h(x,v_{1},v_{2}) \bigr]\\
&\;=\; \max_{v_{1}\in V_{1}}\; \min_{v_{2}\in V_{2}}\;
\bigl[L \varphi^*(x, v_{1}, v_{2}) + h(x,v_{1},v_{2}) \bigr]\,,\\[5pt]
\varphi^*(0) &\;=\; 0
\end{split}
\end{equation}
such that $\beta$ is the value of the game.
\end{theorem}

\begin{proof}
For $(v_{1}, v_{2}) \in {\mathcal M}_{1} \times {\mathcal M}_{2}$, define
$$
J_{\alpha}(x, v_{1}, v_{2}) \;\df\;
\Exp_{x} \biggl[\int^{\infty}_{0} \E^{- \alpha t}
h\bigl(X(t), v_{1}(X(t)), v_{2}(X(t))\bigr)
\,\D{t} \biggr]\,, \quad x \in \mathbb{R}^d\;,
$$
where $X$ is a solution to \eqref{mainsde} corresponding to
$(v_{1}, v_{2})\in {\mathcal M}_{1} \times {\mathcal M}_{2}$.
Hence from \eqref{rep-discounted3}, we have
$$
 \psi_{\alpha}(x) \;=\; J_{\alpha}(x, v^{\alpha}_{1}, v^{\alpha}_{2}) \,,
$$
where $(v^{\alpha}_{1}, v^{\alpha}_{2}) \in {\mathcal M}_{1} \times {\mathcal M}_{2}$
is a pair of measurable outer maximizing and outer minimizing
selectors of \eqref{Isaac-discounted}.
Using (A3$\mspace{2mu}^\prime$), it is easy to see that
$(v^{\alpha}_{1}, v^{\alpha}_{2})$
is a pair of stable stationary Markov controls.
Hence by the arguments in the proof of
\cite[Theorem~3.7.4, pp.~128--131]{AriBorkarGhosh}, we have the following estimates:
\begin{multline}\label{ergodicestimate1}
\|\psi_{\alpha} - \psi_{\alpha}(0) \|_{W^{2, p}(B_{R})}\;\le\;
\frac{K_3}{\eta[v^{\alpha}_{1}, v^{\alpha}_{2}](B_{R})}
\biggl( \frac{\beta[v^{\alpha}_{1}, v^{\alpha}_{2}]}
{\eta[v^{\alpha}_{1}, v^{\alpha}_{2}](B_{R})}\\[5pt]
 + \max_{(x, v_{1}, v_{2}) \in B_{4R} \times V_{1} \times V_{2}}h(x,v_{1},v_{2}) 
 \biggr) \,,
\end{multline}
\begin{equation}\label{ergodicestimate2}
\sup_{x \in B_{R})} \alpha \psi_{\alpha}(x) \;\le\; K_3
\biggl( \frac{\beta[v^{\alpha}_{1}, v^{\alpha}_{2}]}
{\eta[v^{\alpha}_{1}, v^{\alpha}_{2}](B_{R})} +
\max_{(x, v_{1}, v_{2}) \in B_{4R} \times V_{1} \times V_{2}}h(x,v_{1},v_{2})
\biggr) \,,
\end{equation}
where $\eta[v^{\alpha}_{1}, v^{\alpha}_{2}]$ is the unique invariant probability
measure of the process \eqref{mainsde} corresponding to
$(v^{\alpha}_{1}, v^{\alpha}_{2})$ and
$$
\beta[v^{\alpha}_{1}, v^{\alpha}_{2}] \;\df\;\int_{\mathbb{R}^d}
h\bigl(x, v^{\alpha}_{1}(x), v^{\alpha}_{2}(x)\bigr)\,
\eta[v^{\alpha}_{1}, v^{\alpha}_{2}](\D{x}) \;.
$$
It follows from \cite[Corollary 3.3.2, p.~97]{AriBorkarGhosh} that
\begin{equation}\label{ergodicestimate3}
\sup_{\alpha > 0}\;\beta[v^{\alpha}_{1}, v^{\alpha}_{2}] \;< \; \infty \;.
\end{equation}
Also from \cite[(2.6.9a); p.~69 and (3.3.9); p.~97]{AriBorkarGhosh} it follows that
\begin{equation}\label{ergodicestimate4}
\inf_{\alpha > 0}\;\eta[v^{\alpha}_{1}, v^{\alpha}_{2}](B_{R}) \;>\; 0 \;.
\end{equation}
Combining \eqref{ergodicestimate1}--\eqref{ergodicestimate4},
we have
\begin{equation}\label{ergodicestimate5}
\begin{split}
\|\psi_{\alpha} - \psi_{\alpha}(0) \|_{W^{2, p}(B_{R})} & \;\le\; K_4\,,\\[5pt]
\sup_{x \in B_{R}}\; \alpha \psi_{\alpha}(x) & \;\le\; K_4\,,
\end{split}
\end{equation}
where $K_4 > 0$ is a constant independent of $\alpha > 0$.

Define
$$
\Bar{\psi}_{\alpha}(x) \;\df\;\psi_{\alpha}(x) - \psi_{\alpha}(0)\,,
\quad x \in \mathbb{R}^d \,.
$$
In view of \eqref{ergodicestimate5}, one can use the arguments in
\cite[Lemma~3.5.4, pp.~108--109]{AriBorkarGhosh}
to show that along some sequence $\alpha_n \downarrow 0$,
$\alpha_n \psi_{\alpha}(0)$ converges to a constant $\varrho$ and
$\Bar{\psi}_{\alpha_n}$ converges uniformly on compact sets to a function
$\varphi^{*} \in C^{2}(\mathbb{R}^d)$, where the pair
$(\varrho, \, \varphi^{*})$ is a solution to the p.d.e.
\begin{align*}
\varrho &\;=\; \min_{v_{2}\in V_{2}}\; \max_{v_{1}\in V_{1}}\;
\bigl[L \varphi^{*}(x, v_{1}, v_{2}) + h(x,v_{1},v_{2}) \bigr]\;,\\[5pt]
\varphi^{*}(0) &\;=\; 0\;.
\end{align*}
Moreover, using the Isaac's condition, it follows that
$(\varrho, \varphi^{*}) \in \mathbb{R} \times C^2(\mathbb{R}^d)$
satisfies \eqref{Isaac}.

We claim that $\varphi^{*} \in o(\mathcal{V})$, i.e.,
$\frac{\varphi^{*}(x)}{\mathcal{V}(x)}\to 0$ as $\|x\|\to\infty$.
To prove the claim let
$(v^{*}_{1}, v^{*}_{2})\in{\mathcal M}_{1}\times{\mathcal M}_{2}$
be a pair of measurable outer maximizing and
outer minimizing selectors of \eqref{Isaac} corresponding to $\varphi^{*}$.
Let $X$ be the solution to
\eqref{mainsde} under the control $(v^{*}_{1}, v^{*}_{2})$.
Then by an application of the It\^o--Dynkin formula and the help of Fatou's lemma,
we can show that for all $x \in \mathbb{R}^d$
\begin{equation}\label{appendixeq1}
\varphi^{*}(x) \;\ge\;\Exp_{x}\biggl[\int^{\breve{\tau}_{r}}_{0}
\Bigl(h\bigl(X(t), v^{*}_{1}(X(t)), v^{*}_{2}(X(t))\bigr) - \varrho\Bigr) \,\D{t}
\biggr] + \min_{ \|y \| \,=\,r} \varphi^{*}(y) \,,
\end{equation}
where
$$\breve{\tau}_{r} \;=\;\inf\; \{ t \ge 0 : \|X(t)\| \le r \}\,.$$
Let $v^{\alpha}_{1}\in{\mathcal M}_{1}$ be a measurable outer maximizing selector in
\eqref{Isaac-discounted}.
Then the function $\psi_{\alpha} \in C^{2, r}(\mathbb{R}^d)$ given in
Theorem~\ref{thm2.1} satisfies the p.d.e.
\begin{equation}\label{appendixeq2}
\alpha \psi_{\alpha} \;=\;\min_{v_{2} \in V_{2}}\;
\bigl[L \psi_{\alpha}(x, v^{\alpha}_{1}(x), v_{2})
 + h(x, v^{\alpha}_{1}(x), v_{2}) \bigr] \,.
\end{equation}
Let $X$ be the solution to \eqref{mainsde} under the control
$(v^{\alpha}_{1}, v_{2})$, with $v_{2} \in \mathcal{U}_{2}$,
and initial condition $x \in \mathbb{R}^d$.
Then by applying the It\^o--Dynkin formula to
$\E^{-\alpha t} \psi_{\alpha}(X(t))$ and using \eqref{appendixeq2}, we obtain
\begin{equation*}
\Exp_{x} \bigl[\E^{- \alpha (\breve{\tau}_{r} \wedge \tau_{R})}
\psi_{\alpha}(X(\breve{\tau}_{r} \wedge \tau_{R})) \bigr]
- \psi_{\alpha}(x)
\ge\;- \Exp_{x} \biggl[\int^{\breve{\tau}_{r} \wedge \tau_{R}}_{0}
h\bigl(X(t),v^{\alpha}_{1}(X(t)),v_{2}(t)\bigr)\,\D{t} \biggr] \,,
\end{equation*}
which we write as
\begin{equation}\label{appendixeq3}
\psi_{\alpha}(x) \;\le\; \Exp_{x} \biggl[\int^{\breve{\tau}_{r}}_{0}
h\bigl(X(t),v^{\alpha}_{1}(X(t)),v_{2}(t)\bigr)\,\D{t} \biggr]
+ \Exp_{x} \bigl[\E^{- \alpha (\breve{\tau}_{r} \wedge \tau_{R})}
\psi_{\alpha}(X(\breve{\tau}_{r}\wedge \tau_{R})) \bigr] \,.
\end{equation}
Using \cite[Remark A.3.8, p.~310]{AriBorkarGhosh}, it follows that
\begin{equation}\label{appendixeq4}
\Exp_{x} \bigl[\E^{- \alpha \tau_{R}}\psi_{\alpha}(X(\tau_{R}))
I \{\breve{\tau}_{r} \ge \tau_{R}\} \bigr] \;\le\;
\Exp_{x} \bigl[\E^{- \alpha \tau_{R}}\psi_{\alpha}(X( \tau_{R})) \bigr]
\xrightarrow[R \to \infty]{} 0\,.
\end{equation}
Hence from \eqref{appendixeq3} and \eqref{appendixeq4}, we obtain
$$
\psi_{\alpha}(x) \;\le\; \Exp_{x} \biggl[\int^{\breve{\tau}_{r}}_{0}
h\bigl(X(t),v^{\alpha}_{1}(X(t)),v_{2}(t)\bigr)\,\D{t} \biggr]
+ \Exp_{x} \bigl[\E^{- \alpha \breve{\tau}_{r} }
\psi_{\alpha}(X(\breve{\tau}_{r})) \bigr] \,.
$$
Therefore,
\begin{align*}
\Bar{\psi}_{\alpha}(x) &\;\le\; \Exp_{x}\biggl[\int^{\breve{\tau}_{r}}_{0}
h\bigl(X(t),v^{\alpha}_{1}(X(t)),v_{2}(t)\bigr) \,\D{t} \biggr]
+ \Exp_{x}\bigl[\E^{- \alpha \breve{\tau}_{r} }\psi_{\alpha}(X(\breve{\tau}_{r}))
- \psi_{\alpha}(0)\bigr]\\[5pt]
&\;=\; \Exp_{x} \biggl[\int^{\breve{\tau}_{r}}_{0}
\Bigl(h\bigl(X(t),v^{\alpha}_{1}(X(t)),v_{2}(t)\bigr) - \varrho\Bigr)\,\D{t}\biggr]
+ \Exp_{x}\bigl[\psi_{\alpha}(X(\breve{\tau}_{r})) - \psi_{\alpha}(0)\bigr] \\[5pt]
& \mspace{200mu}+ \Exp_{x}\bigl[\alpha^{-1}(1- \E^{-\alpha \breve{\tau}_{r}})
\bigl(\varrho - \alpha \psi_{\alpha}(X(\breve{\tau}_{r}))\bigr)\bigr]\\[5pt]
&\;\le\; \Exp_{x} \biggl[\int^{\breve{\tau}_{r}}_{0}
\Bigl(h\bigl(X(t),v^{\alpha}_{1}(X(t)),v_{2}(t)\bigr) - \varrho\Bigr)\,\D{t}\biggr]
+ M(r)
\\[5pt]
 &\mspace{200mu} +
 \Exp_{x} \bigl[\breve{\tau}_{r}] \sup_{ \|y \|\,=\, r}\;
 \bigl| \varrho- \alpha \psi_{\alpha}(y)\bigr|\\[5pt]
&\;\le\; \sup_{v_{1} \in {\mathcal M}_{1}}\;
\Exp_{x} \biggl[\int^{\breve{\tau}_{r}}_{0}
\Bigl(h\bigl(X(t),v_{1}(X(t)),v_{2}(t)\bigr) - \varrho\Bigr)\,\D{t} \biggr]\\[5pt]
&\mspace{200mu} + M(r) + \sup_{ \|y \|\,=\, r}\;
 \bigl| \varrho- \alpha \psi_{\alpha}(y)\bigr|
\sup_{v_{1} \in {\mathcal M}_{1}}\; \Exp_{x} [ \breve{\tau}_{r}]
\end{align*}
for some nonnegative constant $M(r)$ such that $M(r)\to0$ as $r\downarrow0$.
Next from the definition of $\varphi^{*}$, by letting $\alpha \downarrow 0$ along
the sequence given in the proof of Theorem~\ref{thm2.2}, we obtain
\begin{equation}\label{appendixeq6}
\varphi^{*}(x) \;\le\;\sup_{v_{1} \in {\mathcal M}_{1}}\;
\Exp_{x} \biggl[\int^{\breve{\tau}_{r}}_{0}
\Bigl(h\bigl(X(t),v_{1}(X(t)),v_{2}(t)\bigr) - \varrho\Bigr)\,\D{t} \biggr] + M(r) \,.
\end{equation}
By combining \eqref{appendixeq1} and \eqref{appendixeq6}, the result follows by
\cite[Lemma~3.7.2, p.~125]{AriBorkarGhosh}.
This completes the proof of the claim.

Let $(v^{*}_{1}, v^{*}_{2})\in{\mathcal M}_{1}\times{\mathcal M}_{2}$
be a pair of measurable outer maximizing and minimizing selectors
in \eqref{Isaac} corresponding to $\varphi^{*}$.
Then $(\varrho, \varphi^{*})$ satisfies the p.d.e.
$$
\varrho \;= \; \max_{v_{1} \in V_{1}}\; \bigl[L \varphi^{*}
\bigl(x, v_{1}, v^{*}_{2}(x)\bigr)
+ h\bigl(x, v_{1}, v^{*}_{2}(x)\bigr) \bigr] \,.
$$
Let $v_{1}\in \mathcal{U}_{1}$ and $X$ be the process
in \eqref{mainsde} under the control
$(v_{1}, v^{*}_{2})$ and initial condition $x \in \mathbb{R}^d$.
By applying the It\^o--Dynkin formula, we obtain
$$
\Exp_{x} \bigl[\varphi^{*}(X(t \wedge \tau_{R})) \bigr] - \varphi^{*}(x) \;\le\;
- \Exp_{x} \biggl[\int^{t \wedge \tau_{R}}_{0}
\Bigl(h\bigl(X(t), v_{1}(t), v^{*}_{2}(X(t))\bigr) - \varrho\Bigr) \,\D{t}\biggr].
$$
Hence
$$
\varrho\,t \;\ge\; \Exp_{x} \biggl[\int^{t \wedge \tau_{R}}_{0}
h\bigl(X(t), v_{1}(t), v^{*}_{2}(X(t))\bigr) \,\D{t}\biggr] +
 \Exp_{x} \bigl[\varphi^{*}(X(t \wedge \tau_{R})) \bigr] - \varphi^{*}(x)
$$
for all $t \ge 0$.
Using Fatou's lemma and \cite[Lemma~3.7.2, p.~125]{AriBorkarGhosh}, we obtain
$$
\varrho\,t \,\ge\, \Exp_{x} \biggl[\int^{t}_{0}
h\bigl(X(t), v_{1}(t),v^{*}_{2}(X(t))\bigr)
\,\D{t}\biggr] +
 \Exp_{x} \bigl[\varphi^{*}(X(t)) \bigr] - \varphi^{*}(x)\,,\quad t \ge 0 \,.
$$
Dividing by $t$ and taking limits again
using \cite[Lemma~3.7.2, p.~125]{AriBorkarGhosh},
we obtain
$$
\varrho \;\ge\; \liminf_{t \to \infty}\; \frac{1}{t}\;
\Exp_{x} \biggl[\int^{t}_{0} h\bigl(X(t), v_{1}(t),v^{*}_{2}(X(t))\bigr) \,\D{t}
\biggr].
$$
Since $v_{1} \in \mathcal{U}_{1}$ was arbitrary, we have
\begin{align}\label{ergodicrep1}
\varrho & \;\ge\; \sup_{v_{1} \in\, \mathcal{U}_{1}}\; \liminf_{t \to \infty}\;
\frac{1}{t}\; \Exp_{x} \biggl[\int^{t}_{0}
h\bigl(X(t),v_{1}(t),v^{*}_{2}(X(t))\bigr) \,\D{t}\biggr]\notag\\[5pt]
& \;\ge\; \inf_{v_{2}\in\, \mathcal{U}_{2}}\; \sup_{v_{1} \in\, \mathcal{U}_{1}}\;
\liminf_{t \to \infty}\;
\frac{1}{t}\;
\Exp_{x} \biggl[\int^{t}_{0} h\bigl(X(t), v_{1}(t), v_{2}(t)\bigr) \,\D{t}
\biggr]\;.
\end{align}
The pair $(\varrho, \varphi^{*})$ also satisfies the p.d.e.
$$
\varrho \;=\;\min_{v_{2} \in V_{2}}\;
\bigl[L \varphi^{*}(x, v^{*}_{1}(x), v_{2}) + h(x, v^{*}_{1}(x), v_{2}) \bigr] \,.
$$
Let $v_{2} \in \mathcal{U}_{2}$ and $X$ be the
process in \eqref{mainsde} corresponding to
$(v^{*}_{1}, v_{2})$ and initial condition $x \in \mathbb{R}^d$.
By applying the It\^o--Dynkin formula, we obtain
$$
\Exp_{x} \bigl[\varphi^{*}(X(t \wedge \tau_{R})) \bigr] - \varphi^{*}(x) \;\ge\;
- \Exp_{x} \biggl[\int^{t \wedge \tau_{R}}_{0}
\Bigl(h\bigl(X(t),v^{*}_{1}(X(t)),v_{2}(t)\bigr) - \varrho\Bigr) \,\D{t}\biggr]\,.
$$
Hence
$$
\varrho\, \Exp_{x}[t \wedge \tau_{R}] \;\le\; \Exp_{x} \biggl[\int^{t}_{0}
h\bigl(X(t),v^{*}_{1}(X(t)),v_{2}(t)\bigr) \,\D{t}
+\varphi^{*}(X(t \wedge \tau_{R})) \biggr] - \varphi^{*}(x) \,.
$$
Next, by letting $R \to \infty$ and
using the dominated convergence theorem for the l.h.s.\ and
\cite[Lemma~3.7.2, p.~125]{AriBorkarGhosh} for the r.h.s., we obtain
$$
\varrho\,t \;\le\; \Exp_{x} \biggl[\int^{t}_{0}
h\bigl(X(t),v^{*}_{1}(X(t)),v_{2}(t)\bigr) \,\D{t}\biggr]
+ \Exp_{x} \bigl[\varphi^{*}(X(t)) \bigr] - \varphi^{*}(x) \,.
$$
Also by \cite[Lemma~3.7.2, p.~125]{AriBorkarGhosh}, we obtain
$$
\varrho \;\le\; \liminf_{t \to \infty}\; \frac{1}{t}\;
\Exp_{x} \biggl[\int^{t}_{0}
h\bigl(X(t),v^{*}_{1}(X(t)),v_{2}(t)\bigr) \,\D{t}\biggr] \,.
$$
Since $v_{2} \in \mathcal{U}_{2}$ was arbitrary, we have
\begin{align}\label{ergodicrep2}
\varrho &\;\le\; \inf_{v_{2}\in\, \mathcal{U}_{2}}\;
\liminf_{t \to \infty}\; \frac{1}{t}\;
\Exp_{x}\biggl[\int^{t}_{0} h\bigl(X(t),v^{*}_{1}(X(t)),v_{2}(t)\bigr)\,\D{t}
\biggr]
\notag\\[5pt]
&\;\le\; \sup_{v_{1}\in\, \mathcal{U}_{1}}\;
\inf_{v_{2}\in\, \mathcal{U}_{2}}\; \liminf_{t \to \infty}\;
\frac{1}{t}\;
\Exp_{x} \biggl[\int^{t}_{0} h\bigl(X(t), v_{1}(t), v_{2}(t)\bigr) \,\D{t}
\biggr] \;.
\end{align}
Combining \eqref{ergodicrep1} and \eqref{ergodicrep2}, we obtain
\begin{equation*}
\begin{split}
\varrho &\;=\; \inf_{v_{2}\in\, \mathcal{U}_{2}}\;\sup_{v_{1}\in\, \mathcal{U}_{1}}\;
\liminf_{t \to \infty}\;
 \frac{1}{t}\;
\Exp_{x} \biggl[\int^{t}_{0} h\bigl(X(t),v_{1}(t),v_{2}(t)\bigr) \,\D{t} \biggr]
\\[5pt]
&\;=\; \sup_{v_{1}\in\, \mathcal{U}_{1}}\;\inf_{v_{2}\in\, \mathcal{U}_{2}}\;
\liminf_{t \to \infty}\;\frac{1}{t}\;
\Exp_{x}\biggl[\int^{t}_{0} h\bigl(X(t), v_{1}(t), v_{2}(t)\bigr)\,\D{t}\biggr] \;,
\end{split}
\end{equation*}
i.e. $\varrho = \beta$, the value of the game. This completes the proof.
\end{proof}

\begin{remark}
Using Theorem~\ref{thm2.2}, one can
easily prove that any pair of measurable outer maximizing
and outer minimizing selectors
of \eqref{Isaac-discounted} is a saddle point equilibrium for the
stochastic differential game with state dynamics given by \eqref{mainsde}
and with the ergodic criterion under the running payoff function $h$.
\end{remark}

The following corollary, stated here without proof, follows along the
lines of the proof of \cite[Theorem~3.7.12]{AriBorkarGhosh}.

\begin{corollary}
The solution $\varphi^{*}$ has the stochastic representation
\begin{align*}
\varphi^{*}(x) \;&=\;\lim_{r\downarrow0}\;\sup_{v_{1} \in {\mathcal M}_{1}}\;
\inf_{v_{2} \in {\mathcal M}_{2}}\;
\Exp_{x} \left[\int^{\breve{\tau}_{r}}_{0}
\Bigl(h\bigl(X(t),v_{1}(X(t)),v_{2}(X(t))\bigr) - \beta\Bigr)\,\D{t} \right]
\nonumber\\[5pt]
\;&=\;\lim_{r\downarrow0}\; \inf_{v_{2} \in {\mathcal M}_{2}}\;
\sup_{v_{1} \in {\mathcal M}_{1}}\;
\Exp_{x} \left[\int^{\breve{\tau}_{r}}_{0}
\Bigl(h\bigl(X(t),v_{1}(X(t)),v_{2}(X(t))\bigr) - \beta\Bigr)\,\D{t} \right]
\nonumber\\[5pt]
\;&=\;\lim_{r\downarrow0}\;
\Exp_{x} \left[\int^{\breve{\tau}_{r}}_{0}
\Bigl(h\bigl(X(t),v^{*}_{1}(X(t)),v^{*}_{2}(X(t))\bigr) - \beta\Bigr)\,\D{t}\right]\;.
\end{align*}
and is unique in the class of functions that do not grow
faster than $\mathcal{V}$ and vanish at $x=0$.
\end{corollary}

\section{Relative Value Iteration}
We consider the following relative value iteration equation.
\begin{equation}\label{relativevalueiteration}
\begin{split}
\frac{\partial \varphi}{\partial t}(t, x) \;&=\;
\min_{v_{2}\in V_{2}}\; \max_{v_{1}\in V_{1}}\;
\bigl[L \varphi(t, x, v_{1}, v_{2}) + h(x,v_{1},v_{2}) \bigr] \, - \,
\varphi(t, 0)\;, \\[5pt]
\varphi(0, x) \, &= \, \varphi_{0}(x) \;,
\end{split}
\end{equation}
where $\varphi_{0} \in C_{\mathcal{V}}(\mathbb{R}^d) \cap C^2(\mathbb{R}^d)\,$.
This can be viewed as a continuous time continuous state space variant of the
relative value iteration algorithm for Markov decision processes \cite{White}.

Convergence of this relative value iteration scheme is obtained through
the study of the value iteration equation which takes the form
\begin{equation}\label{valueiteration}
\begin{split}
\frac{\partial \overline{\varphi}}{\partial t}(t, x) \;&=\;
\min_{v_{2}\in V_{2}}\; \max_{v_{1}\in V_{1}}\;
\bigl[L \overline{\varphi}(t, x, v_{1}, v_{2}) + h(x,v_{1},v_{2}) \bigr]
\, - \, \beta\;, \\[5pt]
\overline{\varphi}(0, x) \; &= \; \varphi_{0}(x) \;,
\end{split}
\end{equation}
where $\beta$ is the value of the average payoff game in Theorem~\ref{thm2.2}.

Under Assumption~(A3), it is straightforward to show that for each $T>0$
there exists a unique solution
$\overline{\varphi}$ in $C_{\mathcal{V}}([0,T] \times \mathbb{R}^d)
\cap C^{1,2}([0,T] \times \mathbb{R}^d)$ to the
p.d.e.\ \eqref{valueiteration}.

First, we prove the following important estimate which is crucial for the
proof of convergence.

\begin{lemma}
Assume \textup{(A1)--(A3)}.
Then for each $T > 0\,$,
the p.d.e.\ in \eqref{relativevalueiteration} has a unique solution
$\varphi \in C_{\mathcal{V}}([0,T] \times \mathbb{R}^d)
\cap C^{1,2}([0,T] \times \mathbb{R}^d)\,$.
\end{lemma}

\begin{proof}
The proof follows by mimicking the arguments in \cite[Lemma~4.1]{AriBorkar},
using the following estimate
\begin{equation}\label{estimate-geometric}
\Exp_{x} \bigl[\mathcal{V}(X(t)) \bigr] \;\le\; \frac{k_{0}}{2 k_{1}}
+ \mathcal{V}(x) \E^{ -2 k_{1}t} \,,
\end{equation}
where $X$ is the solution to \eqref{mainsde} corresponding
to any admissible controls
$v_{1}$ and $v_{2}$ and initial condition $x \in \mathbb{R}^d$.
The estimate for $\varphi$ follows from the
arguments in \cite[Lemma~2.5.5, pp.~63--64]{AriBorkarGhosh}, noting that
for all $v_{i}\in \mathcal{U}_{i}$, $i=1,2\,$, we have
\begin{align*}
\int^t_{0}
\Exp_{x}\bigl[h^n\bigl(X(s), v_{1}(s), v_{2}(s)\bigr)\bigr] \,\D{s} & \;\le\;
k_{2}\int^t_{0} \Exp_{x}\bigl[\mathcal{V}(X(s))\bigr] \,\D{s}\\[5pt]
& \;\le\; \frac{k_{2}}{2k_{1}} \bigl(k_{0} t + \mathcal{V}(x)\bigr) \,,
\end{align*}
where $h^n (x, v_{1}, v_{2}) \,\df\,n \wedge h(x,v_{1},v_{2})$ is the
truncation of $h$ at $n\ge0$.
\end{proof}

Next, we turn our attention to the p.d.e.\ in \eqref{valueiteration}.
It is straightforward to show that the solution $\overline{\varphi}$ to
\eqref{valueiteration} also satisfies
\begin{equation}\label{valueiterationbis}
\begin{split}
\frac{\partial \overline{\varphi}}{\partial t}(t, x) \;&=\;
\max_{v_{1}\in V_{1}}\;\min_{v_{2}\in V_{2}}\; 
\bigl[L\, \overline{\varphi}(t, x, v_{1}, v_{2}) + h(x,v_{1},v_{2}) \bigr]
\, - \, \beta\;, \\[5pt]
\overline{\varphi}(0, x) \, &= \, \varphi_{0}(x) \;,
\end{split}
\end{equation}

\begin{definition}
We let $\Bar{v}_{i}:\mathbb{R}_{+}\times\mathbb{R}^{d}\to V_{i}$ for $i=1,2$
be an outer maximizing and an outer minimizing selector of
\eqref{valueiterationbis} and \eqref{valueiteration}, respectively.
For each $t\ge0$ we define the (nonstationary) Markov control
\begin{equation*}
\Bar{v}_{i}^{t}\df\bigl\{\Bar{v}_{i}^{t}(s,\cdot\,)
\;=\;\Bar{v}_{i}(t-s,\cdot\,)\,,\; s\in[0,t]\bigr\}\,.
\end{equation*}
We also let
$\Prob_{x}^{v_{1},v_{2}}$ denote the probability measure and
$\Exp_{x}^{v_{1},v_{2}}$ the expectation operator on the canonical space of the
process under the control $v_{i}\in \mathcal{U}_{i}$, $i=1,2\,$, conditioned on the
process $X$ starting from $x\in\mathbb{R}^{d}$ at $t=0$.
\end{definition}

It is straightforward to show that
the solution $\overline{\varphi}$ of \eqref{valueiteration} satisfies,
\begin{align}\label{visol}
\overline{\varphi}(t,x) \;&=\;
\Exp^{\Bar{v}^{t}_{1},\Bar{v}^{t}_{2}}_{x} \biggl[\int_{0}^{t-s}
\Bigl(h\bigl(X(\tau),\Bar{v}_{1}\bigl(t-\tau, X(\tau)\bigr),
\Bar{v}_{2}\bigl(t-\tau, X(\tau)\bigr)\bigr) - \beta\Bigr)\,\D{\tau}
\nonumber\\
&\mspace{400mu} + \overline{\varphi}\bigl(s,X(t-s)\bigr)\biggr]
\nonumber\\[5pt]
&=\;\inf_{v_{2}\in\, \mathcal{U}_{2}}\; \sup_{v_{1} \in\, \mathcal{U}_{1}}\;
\Exp_{x}^{v_{1},v_{2}} \biggl[\int_{0}^{t-s}
\Bigl(h\bigl(X(\tau),v_{1}(\tau),v_{2}(\tau)\bigr) - \beta\Bigr)\,\D{\tau}
 + \overline{\varphi}\bigl(s,X(t-s)\bigr)\biggr]
\nonumber\\[5pt]
&=\;\sup_{v_{1} \in\, \mathcal{U}_{1}}\;\inf_{v_{2}\in\, \mathcal{U}_{2}}\;
\Exp_{x}^{v_{1},v_{2}} \biggl[\int_{0}^{t-s}
\Bigl(h\bigl(X(\tau),v_{1}(\tau),v_{2}(\tau)\bigr) - \beta\Bigr)\,\D{\tau}
\nonumber\\
&\mspace{400mu} + \overline{\varphi}\bigl(s,X(t-s)\bigr)\biggr]
\end{align}
for all $t\ge s\ge0$.

\begin{lemma}\label{lem3.2}
Assume \textup{(A1)--(A3)}.
For each $\varphi_{0} \in C_{\mathcal{V}}(\mathbb{R}^d) \cap C^2(\mathbb{R}^d)$,
the solution $\overline{\varphi}$ of the p.d.e. \eqref{valueiteration} satisfies
the following estimate
\begin{equation*}
\bigl| \overline{\varphi}(t,x)\,-\,\varphi^*(x)\bigr| \;\le\;
\|\overline{\varphi}(s,\cdot\,)\,-\,\varphi^*\|_{\mathcal{V}}\,
\left( \frac{k_{0}}{2 k_{1}} + \mathcal{V}(x)\, \E^{ -2 k_{1}(t-s)}\right)
\quad \forall x \in \mathbb{R}^d\,,
\end{equation*}
and for all $t\ge s\ge0$,
where $\varphi^*$ is as in Theorem~\ref{thm2.2}.
\end{lemma}

\begin{proof}
Let $v_{1}^*\in{\mathcal M}_{1}$ and $v_{2}^*\in{\mathcal M}_{2}$
be an outer maximizing and outer minimizing selector of \eqref{Isaac},
respectively.
By \eqref{visol} we obtain
\begin{align}
\overline{\varphi}(t,x) - \varphi^*(x)
&\;\le\; \Exp_{x}^{\Bar{v}^{t}_{1},v_{2}^*}
\bigl[\overline{\varphi}\bigl(s,X(t-s)\bigr)
- \varphi^*\bigl(X(t-s)\bigr)\bigr]\label{VIcompA}
\intertext{and}
\varphi^*(x) - \overline{\varphi}(t,x)
&\;\le\; \Exp_{x}^{v_{1}^{*},\Bar{v}^{t}_{2}}\bigl[\varphi^*\bigl(X(t-s)\bigr)
-\overline{\varphi}\bigl(s,X(t-s)\bigr)\bigr]\label{VIcompB}
\end{align}
for all $t\ge s\ge0$.
By \eqref{VIcompA}--\eqref{VIcompB}
we obtain
\begin{equation*}
\bigl| \overline{\varphi}(t,x)\,-\,\varphi^*(x)\bigr| \;\le\;
\sup_{(v_{1},v_{2}) \,\in\, \mathcal{U}_{1}\times\mathcal{U}_{2}}\;
\Exp_{x}^{v_{1},v_{2}}\Bigl[\bigl|\overline{\varphi}\bigl(s,X(t-s)\bigr)
- \varphi^*\bigl(X(t-s)\bigr)\bigr|\Bigr]\,,
\end{equation*}
and an application of \eqref{estimate-geometric} completes the proof. 
\end{proof}

Arguing as in the proof of \cite[Lemma~4.4]{AriBorkar}, we can show the following:

\begin{lemma}\label{lem3.3}
Assume \textup{(A1)--(A3)}.
If $\overline{\varphi}(0, x) = \varphi(0, x) = \varphi_{0}(x)$ for some
$\varphi_{0} \in C_{\mathcal{V}}(\mathbb{R}^d) \cap C^2(\mathbb{R}^d)$, then
\begin{equation*}
 \varphi(t,x) - \varphi(t, 0) \;=\;
 \overline{\varphi}(t, x) - \overline{\varphi}(t, 0)\;,
\end{equation*}
and
\begin{equation*}
 \varphi(t, x) \;=\; \overline{\varphi}(t,x) - \E^{-t}
\int^t_{0}\E^s\,\overline{\varphi}(s, 0) \,\D{s} + \beta (1- \E^{-t})
\end{equation*}
for all $x \in \mathbb{R}^d$ and $t \ge 0$.
\end{lemma}

Convergence of the relative value iteration is asserted in the following
theorem.

\begin{theorem}\label{thm3.1}
Assume \textup{(A1)--(A3)}.
For each $\varphi_{0} \in C_{\mathcal{V}}(\mathbb{R}^d) \cap C^2(\mathbb{R}^d)$,
$\overline\varphi(t, x)$ converges to $\varphi^*(x) + \text{constant}$ and
$\varphi(t, x)$ converges to $\varphi^*(x) + \beta$ as $t \to \infty$.
\end{theorem}

\begin{proof}
By Lemma~\ref{lem3.2}
the map $x\mapsto\overline{\varphi}(t,x)$ is locally bounded, uniformly in $t\ge0$.
It then follows that $\bigl\{\frac{\partial^{2}\overline{\varphi}(t,x)}
{\partial{x_{i}}\partial{x_{j}}}\,,~t\ge1\bigr\}$ are locally H\"older equicontinuous
(see \cite[Theorem~5.1]{Lady}).
Therefore the $\omega$-limit set $\omega(\varphi_{0})$ of any initial condition
$\varphi_{0} \in C_{\mathcal{V}}(\mathbb{R}^d) \cap C^2(\mathbb{R}^d)$
is a nonempty compact subset of
$C_{\mathcal{V}}(\mathbb{R}^d) \cap C_{loc}^2(\mathbb{R}^d)$.

To simplify the notation we define
\begin{equation*}
\varPhi_{t}(x) \;\df\; \overline{\varphi}(t,x) - \varphi^{*}(x)\,,\quad (t,x)\in
\mathbb{R}_{+}\times\mathbb{R}^{d}\;.
\end{equation*}
By Lemma~\ref{lem3.2}, if $f\in\omega(\varphi_{0})$
then
\begin{equation*}
\limsup_{t\to\infty}\;|\varPhi_{t}(x)|
 \;\le\; \frac{k_{0}}{2 k_{1}}\, \|\varphi_{0}\,-\,\varphi^*\|_{\mathcal{V}}\;.
\end{equation*}
Let $\{t_{n}\,\ n\in\mathbb{N}\}\subset\mathbb{R}_{+}$
be any increasing sequence such that $t_{n}\uparrow\infty$ and
\begin{equation*}
\varPhi_{t_{n}}\to f\,\in\,C_{\mathcal{V}}(\mathbb{R}^d)
\cap C^2(\mathbb{R}^d)\qquad\text{as~}n\to\infty\;.
\end{equation*}
Dropping to a subsequence we assume that $t_{n+1}-t_{n}\uparrow\infty$
as $n\to\infty$.
By construction $f+\varphi^{*}\in\omega(\varphi_{0})$.

We first show that $f$ is a constant.
We define
\begin{equation*}
\Bar{f} \;\df\; \sup_{x\in\mathbb{R}^{d}}\;f(x)\;,
\end{equation*}
and a subsequence $\{k_{n}\}\subset \mathbb{N}$ by
\begin{equation}\label{E.00}
k_{n} \;\df\; \sup\;\Bigl\{k\in\mathbb{N} :
\sup_{x\in B_{k}}\;\varPhi_{t_{n}}(x)\;\le\;\Bar{f}+\frac{1}{k}\Bigr\}\;.
\end{equation}
Since $\varPhi_{t_{n}}$ converges to $f$ uniformly on compact sets as
$n\to\infty$, it follows that $k_{n}\uparrow \infty$ as $n\to\infty$.
Let $D$ be any fixed closed ball centered at the origin such that
$$
\inf_{x\in D^{c}}\;\mathcal{V}(x) \;\ge\;
\frac{2k_{0}\,\|\varphi_{0}\,-\,\varphi^*\|_{\mathcal{V}}}{k_{1}}\,.
$$
It is straightforward to verify using \eqref{estimate-geometric}
that if $X$ is the solution to \eqref{mainsde} corresponding
to any admissible controls
$v_{1}$ and $v_{2}$ and initial condition $x \in \mathbb{R}^d$
then there exists $T_{0}<\infty$ depending only on $x$, such that
\begin{equation}\label{E.001}
\Prob_{x}\bigl(X_{t}\in D) \;\ge\; \frac{1}{2}\qquad
\forall x\in \mathbb{R}^{d}\,,\quad\forall t\ge T_{0}(x)\;.
\end{equation}
By the standard estimates of hitting probabilities for diffusions
(see \cite[Lemma~1.1]{Gruber}) for any $r>0$ there exists a
constant $\gamma>0$ depending only on $r$ and $D$, such that
with $B_{r}(y)$ denoting the open ball of radius $r$ centered at $y\in\mathbb{R}^{d}$
we have
\begin{equation}\label{E.002}
\Prob_{x}\bigl(X_{t}\in B_{r}(y)\bigr) \;\ge\; \gamma \qquad \forall t\in[0,1]\,,
\quad\forall x,y\in D\;.
\end{equation}
Let $\mathrm{I}_{A}(\,\cdot\,)$ denote the indicator
function of a set $A\subset\mathbb{R}^{d}$.
An equivalent statement to \eqref{E.002}
is that if $g:D\to\mathbb{R}_{+}$ is a H\"older continuous function
then there exists a continuous function $\Gamma:\mathbb{R}_{+}\to\mathbb{R}_{+}$,
satisfying $\Gamma(z)>0$ for $z>0$ and depending only
on $D$ and the H\"older constant of $g$, such that
\begin{equation}\label{E.003}
\Exp_{x}\bigl(g(X_{t})\,\mathrm{I}_{D}(X_{t})\bigr) \;\ge\;
\Gamma\Bigl(\max_{y\in D}\;g(y)\Bigr)
\qquad \forall t\in[0,1]\,,\quad \forall x\in D\;.
\end{equation}
Combining \eqref{E.001} and \eqref{E.003} and using the Markov
property, we obtain
\begin{align}\label{E.004}
\Exp_{x}\bigl[g(X_{t})\,\mathrm{I}_{D}(X_{t})\bigr]
\;&\ge\;
\Exp_{x}\Bigl[\Exp_{X_{t-1}}\bigl[g(X_{1})\,\mathrm{I}_{D}(X_{1})\bigr]
\,\mathrm{I}_{D}(X_{t-1})\Bigr]
\nonumber\\[5pt]
&\ge\;
\Gamma\Bigl(\max_{y\in D}\;g(y)\Bigr)\,\Prob_{x}\bigl(X_{t-1}\in D)
\nonumber\\[5pt]
&\ge\;
\frac{1}{2}\,\Gamma\Bigl(\max_{y\in D}\;g(y)\Bigr)
\qquad \forall t\ge T_{0}(x) +1\;.
\end{align}
and for all $x\in\mathbb{R}^{d}$.
Note that if $n$ is sufficiently large, then
$D\subset B_{k_{n}}$ and therefore
the function $x\mapsto \Bar{f}+\frac{1}{k_{n}}- \varPhi_{t_{n}}(x)$
is nonnegative on $D$.
Thus the local H\"older equicontinuity of
$\{\varPhi_{t}\,,\;t>0\}$
(this collection of functions locally share a common H\"older exponent)
allows us to apply \eqref{E.004}
for any fixed $x\in\mathbb{R}^{d}$ to obtain
\begin{multline}\label{E.01}
\Exp_{x}^{\Bar{v}^{t_{n+1}}_{1},v_{2}^*}
\Bigl[\Bigl(\Bar{f}+\frac{1}{k_{n}}
- \varPhi_{t_{n}}\bigl(X(t_{n+1}-t_{n})\bigr)\Bigr)\,
\mathrm{I}_{D}\bigl(X(t_{n+1}-t_{n})\bigr)\Bigr] \\
\ge\; \frac{1}{2}\;\Gamma\Bigl( \Bar{f}+\frac{1}{k_{n}}
- \min_{y\in D} \varPhi_{t_{n}}(y)\Bigr)\,,
\end{multline}
for all $n$ large enough.
For $A\subset\mathbb{R}^{d}$ and $x\in\mathbb{R}^{d}$ we define
\begin{equation*}
\varPsi_{n}(x;A) \;\df\;
\Exp_{x}^{\Bar{v}^{t_{n+1}}_{1},v_{2}^*}
\Bigl[\varPhi_{t_{n}}\bigl(X(t_{n+1}-t_{n})\bigr)\,
\mathrm{I}_{A}\bigl(X(t_{n+1}-t_{n})\bigr)\Bigr]
\end{equation*}
By \eqref{VIcompA}, \eqref{E.00} and \eqref{E.01} we have
\begin{align}\label{E.02}
\varPhi_{t_{n+1}}(x) 
\;&\le\; \Exp_{x}^{\Bar{v}^{t_{n+1}}_{1},v_{2}^*}
\bigl[\varPhi_{t_{n}}\bigl(X(t_{n+1}-t_{n})\bigr)\bigr]
\notag\\[5pt]
&=\; \varPsi_{n}(x;D) + \varPsi_{n}(x;B_{k_{n}}\setminus D)
+ \varPsi_{n}(x;B^{c}_{k_{n}})
\notag\\[5pt]
&\le\;
\Bigl(\Bar{f}+\frac{1}{k_{n}}\Bigr)\,
\Exp_{x}^{\Bar{v}^{t_{n+1}}_{1},v_{2}^*}
\Bigl[\mathrm{I}_{B_{k_{n}}}\bigl(X(t_{n+1}-t_{n})\bigr)\Bigr]
\notag\\
&\mspace{100mu}
- \frac{1}{2}\;\Gamma\Bigl( \Bar{f}+\frac{1}{k_{n}}
- \min_{y\in D} \varPhi_{t_{n}}(y)\Bigr)
+ \varPsi_{n}(x;B^{c}_{k_{n}})
\notag\\[5pt]
&\le\;
\Bar{f}+\frac{1}{k_{n}}
- \frac{1}{2}\;\Gamma\Bigl( \Bar{f}+\frac{1}{k_{n}}
- \min_{y\in D} \varPhi_{t_{n}}(y)\Bigr)
+ \varPsi_{n}(x;B^{c}_{k_{n}})\;.
\end{align}
We claim that $\varPsi_{n}(x;B^{c}_{k_{n}})\to0$ as $n\to\infty$.
Indeed if $X$ is the solution to \eqref{mainsde} corresponding
to any admissible controls
$v_{1}$ and $v_{2}$ and initial condition $x \in \mathbb{R}^d$
then by \eqref{estimate-geometric} we have
\begin{equation}\label{E.03}
\Exp_{x} \bigl[\varPhi_{t}\bigl(X(s)\bigr)\,
\mathrm{I}_{B^{c}_{R}}\bigl(X(s)\bigr)\bigr] \;\le\;
\bigl\|\varPhi_{t}\,\mathrm{I}_{B^{c}_{R}}\bigr\|_{\mathcal{V}}
\left(\frac{k_{0}}{2 k_{1}}
+ \mathcal{V}(x) \E^{ -2 k_{1}s} \right)\,,
\end{equation}
By Lemma~\ref{lem3.2} we have
\begin{equation}\label{E.04}
\bigl\| \varPhi_{t}\,\mathrm{I}_{B^{c}_{R}}\bigr\|_{\mathcal{V}} \;\le\;
\|\varPhi_{0}\|_{\mathcal{V}}\,
\left( \frac{k_{0}}{2 k_{1}\,\inf_{x\in B_{R}^{c}}\;\mathcal{V}(x)}
+ \E^{ -2 k_{1}t}\right)\;.
\end{equation}
It follows by \eqref{E.03}--\eqref{E.04} that
\begin{equation*}
\Exp_{x} \bigl[\varPhi_{t}\bigl(X(s)\bigr)\,
\mathrm{I}_{B^{c}_{R}}\bigl(X(s)\bigr)\bigr]
\xrightarrow[\min\{t,R\}\to\infty]{}0
\end{equation*}
uniformly in $s\ge0$, which proves that
$\varPsi_{n}(x;B^{c}_{k_{n}})\to0$ as $n\to\infty$.
Thus, taking limits as $n\to\infty$ in \eqref{E.02}, we obtain
\begin{equation}\label{E.05}
f(x) \;\le\; \Bar{f}
- \frac{1}{2}\;\Gamma\Bigl(\Bar{f} - \min_{y\in D} f(y)\Bigr)
\qquad \forall x\in\mathbb{R}^{d}\;.
\end{equation}
Taking the supremum over $x\in\mathbb{R}^{d}$ of the left hand side
of \eqref{E.05} it follows that
$$\Gamma\Bigl(\Bar{f} - \min_{y\in D} f(y)\Bigr)=0$$ which implies that
$f$ is constant on $D$.
Since $D$ was arbitrary if follows that $f$ must be a constant.

We next show that $f$ is unique.
We argue by contradiction.
Suppose that $\varPhi_{t'_{n}}\to f'$ over some increasing sequence
$\{t'_{n}\}$ with $t'_{n}\uparrow\infty$ as $n\to\infty$.
Without loss of generality we assume $t_{n}<t'_{n}<t_{n+1}$ for each $n$.
By \eqref{VIcompA} we have
\begin{equation}\label{E.06}
\varPhi_{t_{n+1}}(x)
\;\le\; \Exp_{x}^{\Bar{v}^{t_{n+1}}_{1},v_{2}^*}
\bigl[\varPhi_{t'_{n}}\bigl(X(t_{n+1}-t'_{n})\bigr)\bigr]\;,
\end{equation}
and taking limits as $n\to\infty$ in \eqref{E.06} we obtain $f\le f'$.
Reversing the roles of $f$ and $f'$, shows that $f=f'$.

By Lemma~\ref{lem3.3} we have
$$
\varphi(t,x) \;=\; \overline{\varphi}(t,x)
+\int^t_{0} \E^{s-t} (\beta - \overline{\varphi}(s, 0)) \,\D{s} \,.
$$
Hence, since $\overline\varphi(t,x)$ converges
to $\varphi^*(x) + f$,
we obtain that
$\varphi(t, x) \to \varphi^*(x) + \beta$ as $t \to \infty$. 
\end{proof}

\section{Risk-Sensitive Control}
In this section, we apply the results from Section 3
to study the convergence of a relative value iteration scheme for
the risk-sensitive control problem which is described as follows.
Let $U$ be a compact metric space and $V \;=\;{\mathcal P}(U)$
denote the space of all probability measures on $U$ with Prohorov topology.
We consider the risk-sensitive control problem with state
equation given by the controlled s.d.e. (in relaxed form)
\begin{equation}\label{risksensde}
\D{X}(t) \;=\;b(X(t), v(t)) \,\D{t} + \sigma(X(t))\, \D{W}(t)\,,
\end{equation}
and payoff criterion
$$
J(x, v) \;\df\; \liminf_{T \to \infty}\;\frac{1}{T}\;
\ln \Exp_{x} \biggl[\exp\biggl(\int^T_{0} h(X(t), v(t)) \,\D{t}\biggr)
\Bigm| X(0) = x\biggr]
\,.
$$
This is called the risk-sensitive payoff because in some sense it is sensitive to
higher moments of the running cost and not merely its mean,
thus capturing `risk' in the sense understood in economics \cite{Whittle}.

All processes in \eqref{risksensde} are defined in a common probability space
$(\varOmega, \, {\mathcal F}, \, P)$
which is assumed to be complete.
The process $W$ is an $\mathbb{R}^d$-valued standard Wiener process
which is independent of
the initial condition $X_{0}$ of \eqref{mainsde}.
The control $v$ is a $V$-valued process which is
jointly measurable in $(t, \omega) \in [0,\infty) \times\varOmega$
and non-anticipative, i.e., for $s < t$, $W(t) - W(s)$ is independent of
${\mathcal F}_s \;\df$ the completion of $\sigma(X_{0}, v(r), W(r), r \le s)\,$.
We denote the
set of all such controls (admissible controls) by $\mathcal{U}$.

\smallskip
\noindent\emph{Assumptions on the Data:}
We assume the following properties for the coefficients $b$ and $\sigma$:

\begin{itemize}
\item[(B1)]
The functions $b$ and $\sigma$ are continuous and bounded, and also
Lipschitz continuous in $x\in\mathbb{R}^d$ uniformly over $v\in V$.
Also $(\sigma \sigma\transp)^{-1}$ is Lipschitz continuous.

\smallskip
\item[(B2)]
For each $R>0$ there exists a constant $\kappa(R) > 0$ such that
$$
z\transp a(x) z \;\ge\; \kappa(R) \|z\|^2
\qquad \text{for~all~} \|x\|\le R \text{~and~} z \in \mathbb{R}^d \,,
$$
where $a\;\df\;\sigma \sigma\transp$.
\end{itemize}

\smallskip
\noindent\emph{Asymptotic Flatness Hypothesis:}
We assume the following property:

\begin{itemize}
\item[(B3)]
\begin{itemize}
\item[(i)]
There exists a $c > 0$ and a positive definite matrix $Q$ such that
for all $x$, $y \in \mathbb{R}^d$ with $x \neq y$, we have
\begin{multline*}
2 \bigl(b(x, v) - b(y, v)\bigr)\transp Q (x-y) + \mathrm{tr}
\Bigl(\bigl(\sigma(x)-\sigma(y)\bigr)\bigl(\sigma(x)-\sigma(y)\bigr)\transp Q\Bigr)
\\[5pt]
- \frac{\bigl\|
\bigl(\sigma(x)-\sigma(y)\bigr)\transp Q (x-y)\bigr\|^2}
{(x-y)\transp Q(x-y)}
\;\le\; -c\, \|x-y\|^2\,.
\end{multline*}
\item[(ii)]
Let $\mathrm{Lip}(f)$ denote the Lipschitz constant of a Lipschitz
continuous function $f$. Then
$$
2\, \|\sigma \sigma\transp\|^2_{\infty}\, \mathrm{Lip}(h)\,
\mathrm{Lip}\bigl((\sigma\sigma\transp)^{-1}\bigr) \, \le\, c^2 \;.
$$
\end{itemize}
\end{itemize}

The asymptotic flatness hypothesis was first introduced by \cite{Basak-Bhattacharya}
for the study of ergodicity in degenerate diffusions and is a little more general
than the condition introduced by \cite{Fleming-McEneaney} in risk-sensitive
control to facilitate the analysis of the corresponding HJB equation,
which is our motivation as well.
An important consequence of this condition is that if we fix a
non-anticipative control process and consider two diffusion processes with this
control differing only in their initial conditions, they approach
each other in mean at an exponential rate \cite[Lemma~7.3.4]{AriBorkarGhosh}. 
This ensures a bounded gradient for the solution of the HJB equation, 
a key step in the analysis of its well-posedness.

We quote the following result from \cite[Theorems~2.2--2.3]{BorkarSuresh}:
\begin{theorem}\label{thm4.1}
Assume \textup{(B1)--(B3)}. The p.d.e.
\begin{equation}\label{Isaac-risk}
 \begin{split}
 \beta &\;=\;
 \min_{v\in V}\; \max_{w \in \mathbb{R}^d}\;
 \Bigl[\Tilde{L} \varphi^*(x, w, v) + h(x,v)
- \tfrac{1}{2}\, w\transp \bigl(a^{-1}(x)\bigr) w \Bigr]\\
&\;=\;
\max_{w \in \mathbb{R}^d}\; \min_{v\in V}\;
\Bigl[\Tilde{L} \varphi^*(x, w, v) + h(x,v)
- \tfrac{1}{2}\, w\transp \bigl(a^{-1}(x)\bigr) w \Bigr]\;, \\[5pt]
\varphi^*(0) &\;=\; 0 \,,
 \end{split}
\end{equation}
where
\begin{equation*}
 \Tilde{L} f(x, w, v) \;\df\;(b(x, v) + w) \cdot \nabla f(x) + \frac{1}{2}\,
\mathrm{tr}\bigl(a(x) \nabla^2 f(x)\bigr)\,,\quad
 f\in C^{2}(\mathbb{R}^d)\;,
\end{equation*}
has a unique solution
$(\beta, \varphi^*) \in \mathbb{R} \times C^2(\mathbb{R}^d) \cap o(\|x\|)$.
Moreover, $\beta$ is the value of the risk-sensitive control problem and
any measurable outer minimizing selector
in \eqref{Isaac-risk} is risk-sensitive optimal.
Also in \eqref{Isaac-risk}, the supremum can be restricted
to a closed ball $\Tilde{V} = \overline{B_{R}}$ for
$$
 R \;\df \; \frac{\mathrm{Lip}(h)}{c}
 + \frac{\mathrm{Lip}\bigl((\sigma\sigma\transp)^{-1}\bigr) K^2}{2 \sqrt{c}} \ ,
$$
where $K$ is the smallest positive root (using \textup{(B3)$\,$(ii)}) of
$$
 \frac{\sqrt{c}}{2}\, \|\sigma \sigma\transp\|_{\infty}\,
\mathrm{Lip}\bigl((\sigma\sigma\transp)^{-1}\bigr)\, x^2 - c^{\nicefrac{5}{4}} x +
\mathrm{Lip} (h) \|\sigma \sigma\transp\|_{\infty} \;=\; 0 \, .
$$
\end{theorem}

For the stochastic differential game in \eqref{Isaac-risk} we
consider the following relative value iteration equation:
\begin{align*}
\frac{\partial \varphi}{\partial t}(t, x) &\;=\;
\min_{v \in V}\; \max_{w \in \Tilde{V}}\;
\Bigl[\Tilde{L} \varphi(t, x, w, v) + h(x,v)
- \tfrac{1}{2}\, w\transp \bigl(a^{-1}(x)\bigr) w \Bigr] \, - \, \varphi(t, 0)\;,
\\[5pt]
 \varphi(0, x) &\;=\; \varphi_{0}(x) \;,
\end{align*}
where $\varphi_{0} \in C_{\mathcal{V}}(\mathbb{R}^d) \cap C^2(\mathbb{R}^d)$ with
$$\mathcal{V}(x) \;=\;\frac{(x\transp Q x)^{1 + \alpha}}
{\varepsilon + (x\transp Q x)^{\nicefrac{1}{2}}}\,,$$
for some positive constants $\varepsilon$ and $\alpha$.
Here note that Assumption~(B3) implies
Assumption~(A3) of Section 2 for the Lyapunov function $\mathcal{V}$ given above,
see \cite[equation (7.3.6), p.~257]{AriBorkarGhosh}.

By Theorems~\ref{thm3.1} and \ref{thm4.1} the following holds.

\begin{theorem}\label{thm4.2}
Assume \textup{(B1)--(B3)}.
For each $\varphi_{0} \in C_{\mathcal{V}}(\mathbb{R}^d) \cap C^2(\mathbb{R}^d)$,
$\varphi(t, x)$ converges to $\varphi^*(x) + \beta$ as $t \to \infty$.
\end{theorem}

\medskip
The relative value iteration equation for the
risk-sensitive control problem is given by
\begin{equation}\label{risksensitivervi}
\begin{split}
\frac{\partial \psi}{\partial t}(t,x) \,&=\, \min_{v \in V}\;
\bigl[ L \psi(t, x, v) + (h(x,v) - \ln \psi(t,0)) \psi(t,x)\bigr]\,,\\[5pt]
\psi(0, x) \,&=\, \psi_{0}(x) \;,
\end{split}
\end{equation}
where
$$
 L f(x, v) \;\df\; b(x, v) \cdot \nabla f(x) + \frac{1}{2}\,
 \mathrm{tr}\bigl(a(x)\nabla^2 f(x)\bigr)\,,\quad
 f\in C^{2}(\mathbb{R}^d)\;.
$$

That one has $\ln \psi(t,0)$ instead of $\psi(t,0)$ as the `offset' is only natural,
because we are trying to approximate the \textit{logarithmic} growth rate of the cost.
We have the following theorem:

\begin{theorem}
Let $\psi^*$ be the unique solution in the
class of functions which grow no faster than $\E^{\|x\|^2}$
of the HJB equation for the risk-sensitive control problem given by
\begin{equation*}
\beta \psi^* \,=\, \min_{v \in V}\; \bigl[ L \psi^*(x, v) + h(x,v) \psi^* \bigr]\,,
\quad \psi^*(0) \;=\; 1 \,.
\end{equation*}
Under assumptions \textup{(B1)--(B3)} the solution $\psi(t, x)$ of
the relative value iteration in
\eqref{risksensitivervi} converges as $t\to\infty$ to
$\E^{\beta} \psi^*(x)$ where $\beta$ is the value of
the risk-sensitive control problem given in Theorem~\ref{thm4.1}.
\end{theorem}

\begin{proof}
A straightforward calculation shows that $\psi^* =\; \E^{\varphi^*}$, where
$\varphi^*$ is given in Theorem~\ref{thm4.1}.
Then it easily follows that $\psi(t,x) = \E^{\varphi(t,x)}$, where $\varphi$ is
the solution of the relative value iteration for the stochastic differential game
in \eqref{Isaac-risk}.
From Theorem~\ref{thm4.2}, it follows that
$\psi(t,x) \to \E^{\beta} \psi^*(x)$ as $t \to \infty$, which
establishes the claim.
\end{proof}

\section{Acknowledgement}
The work of Ari Arapostathis was supported in part by the Office of Naval
Research under the
Electric Ship Research and Development Consortium.
The work of Vivek Borkar was supported in part by Grant \#11IRCCSG014
from IRCC, IIT, Mumbai.

\def\cprime{$'$}


\begin{thebibliography}{10}
\providecommand{\url}[1]{{#1}}
\providecommand{\urlprefix}{URL }
\expandafter\ifx\csname urlstyle\endcsname\relax
  \providecommand{\doi}[1]{DOI~\discretionary{}{}{}#1}\else
  \providecommand{\doi}{DOI~\discretionary{}{}{}\begingroup
  \urlstyle{rm}\Url}\fi

\bibitem{AriBorkar}
Arapostathis, A., Borkar, V.S.: A relative value iteration algorithm for
  nondegenerate controlled diffusions.
\newblock SIAM J. Control Optim. \textbf{50}(4), 1886--1902 (2012)

\bibitem{AriBorkarGhosh}
Arapostathis, A., Borkar, V.S., Ghosh, M.K.: Ergodic control of diffusion
  processes, \emph{Encyclopedia of Mathematics and its Applications}, vol. 143.
\newblock Cambridge University Press, Cambridge (2011)

\bibitem{Basak-Bhattacharya}
Basak, G.K., Bhattacharya, R.N.: Stability in distribution for a class of
  singular diffusions.
\newblock Ann. Probab. \textbf{20}(1), 312--321 (1992)

\bibitem{Benes}
Bene{\v{s}}, V.E.: Existence of optimal strategies based on specified
  information, for a class of stochastic decision problems.
\newblock SIAM J. Control \textbf{8}, 179--188 (1970)

\bibitem{BorkarGhosh}
Borkar, V.S., Ghosh, M.K.: Stochastic differential games: occupation measure
  based approach.
\newblock J. Optim. Theory Appl. \textbf{73}(2), 359--385 (1992)

\bibitem{BorkarSuresh}
Borkar, V.S., Suresh~Kumar, K.: Singular perturbations in risk-sensitive
  stochastic control.
\newblock SIAM J. Control Optim. \textbf{48}(6), 3675--3697 (2010)

\bibitem{Fleming-McEneaney}
Fleming, W.H., McEneaney, W.M.: Risk-sensitive control on an infinite time
  horizon.
\newblock SIAM J. Control Optim. \textbf{33}(6), 1881--1915 (1995)

\bibitem{GilbargTrudinger}
Gilbarg, D., Trudinger, N.S.: Elliptic partial differential equations of second
  order, \emph{Grundlehren der Mathematischen Wissenschaften}, vol. 224, second
  edn.
\newblock Springer-Verlag, Berlin (1983)

\bibitem{Gruber}
Gruber, M.: Harnack inequalities for solutions of general second order
  parabolic equations and estimates of their {H}\"older constants.
\newblock Math. Z. \textbf{185}(1), 23--43 (1984)

\bibitem{Lady}
Lady{\v{z}}enskaja, O.A., Solonnikov, V.A., Ural{\cprime}ceva, N.N.: Linear and
  quasilinear equations of parabolic type.
\newblock Translated from the Russian by S. Smith. Translations of Mathematical
  Monographs, Vol. 23. American Mathematical Society, Providence, R.I. (1967)

\bibitem{MT-III}
Meyn, S.P., Tweedie, R.L.: Stability of {M}arkovian processes. {III}.
  {F}oster-{L}yapunov criteria for continuous-time processes.
\newblock Adv. in Appl. Probab. \textbf{25}(3), 518--548 (1993)

\bibitem{White}
White, D.J.: Dynamic programming, {M}arkov chains, and the method of successive
  approximations.
\newblock J. Math. Anal. Appl. \textbf{6}, 373--376 (1963)

\bibitem{Whittle}
Whittle, P.: Risk-sensitive optimal control.
\newblock Wiley-Interscience Series in Systems and Optimization. John Wiley \&
  Sons Ltd., Chichester (1990)

\end{thebibliography}
\end{document}